\documentclass[10pt]{amsart}

\usepackage{tikz}
\usepackage{amssymb,amsmath,latexsym,graphicx}
\usepackage{charter,eucal}
\usepackage{amssymb}
\usepackage[all,cmtip]{xy}
\usepackage{rotating,multicol}
\usetikzlibrary{decorations.markings}

\newtheorem{theorem}{Theorem }[section]

\newtheorem{lemma}[theorem]{Lemma}
\newtheorem{observation}[theorem]{Observation}

\newtheorem{remark}[theorem]{Remark}
\newtheorem{corollary}[theorem]{Corollary}
\newtheorem{proposition}[theorem]{Proposition}
\newtheorem{principle}[theorem]{\textsc{Principle}}

\newcommand{\bt}{\begin{theorem}}
\newcommand{\et}{\end{theorem}}
\newcommand{\bmt}{\begin{maintheorem}}
\newcommand{\emt}{\end{maintheorem}}
\newcommand{\bc}{\begin{corollary}}
\newcommand{\bl}{\begin{lemma}}
\newcommand{\ec}{\end{corollary}}
\newcommand{\el}{\end{lemma}}
\newcommand{\bo}{\begin{observation}}
\newcommand{\eo}{\end{observation}}
\newcommand{\bp}{\begin{proposition}}
\newcommand{\ep}{\end{proposition}}
\newcommand{\br}{\begin{remark}}
\newcommand{\er}{\end{remark}}
\newcommand{\bpr}{\begin{principle}}
\newcommand{\epr}{\end{principle}}

\def\PG{\mathbf{PG}}

\def\AG{\mathbf{AG}}

\def\eop{\hspace*{\fill}$\blacksquare$}

\usepackage{amssymb} %maths
\usepackage{amsmath} %maths
\usepackage[utf8]{inputenc} %useful to type directly diacritic characters
\usepackage{latexsym}

\title {\bf Generalized affine spaces, generalized ovoids and generalized quadrangles}
\author {J. A. Thas \\Ghent University}
\address{Ghent University, Department of Mathematics, Computer Science and Statistics \\Krijgslaan 299, S9, B-9000 Ghent, Belgium}
\email{thas.joseph@gmail.com}
\begin {document}
\maketitle
\begin{abstract}
We consider a partition of the projective space $\PG(m + n - 1, \bf K)$, with {\bf K} any (commutative)  field, into one space of dimension $m - 1$ while all other spaces of the partition have dimension $n - 1$. Characterizations of particular partitions of this type are given. This research is motivated by some interesting problems on translation generalized quadrangles \cite {JT:23}. In the paper we apply our results to generalized ovoids and generalized quadrangles. Finally we give some ideas for further research.


\end {abstract}

\section {introduction}
In a previous paper \cite {JT:23} partitions of a projective space  by subspaces come in. It concerned a particular partition of the projective space $\PG(3n - 1, q)$ over the finite field GF$(q)$ by one subspace of dimension $2n - 1$ and $q^{2n}$ subspaces of dimension $n - 1$. These partitions played a crucial role in the paper.
\par Here we will consider partitions of $\PG(m + n - 1, \bf K)$, with {\bf K} any (commutative) field, by one subspace of dimension $m - 1$ and subspaces of dimension $n - 1$. The projective space $\PG(m + n - 1, \bf K)$ together with the elements of the partition will be called  a {\em space with an affine spread}.
\par In Section 2 we introduce the {\em  regulus condition}, {\em condition (R)} for short, and derive properties of spaces with an affine spread satisfying this condition. In this case the set $S$ of the $(n - 1)$-dimensional spaces will be called a {\em generalized affine space}. That property appears to play a key role in the other sections of the paper.
\par In Section 3 we introduce {\em conditions (T1)} and {\em (T2)} and show equivalencies between them.
\par Section 4 is on projections and Section 5 contains the main results where we determine the structure of generalized affine spaces satisfying condition (Ti), $i \in \lbrace 1,2 \rbrace$. For {\bf K} = GF($q$) they are unique up to isomorphism. Here it becomes clear why the term generalized affine space is used.
\par Sections 6 and 7 are on applications to generalized ovoids and generalized quadrangles.
\par In the last Section 8 we give some ideas for further research.

\section {The regulus condition}
Let {\bf K} be a (commutative) field and let $\PG(m - 1, \bf {K})$ be an $(m - 1)$-dimensional subspace of the $(m + n - 1)$-dimensional projective space $\PG(m + n - 1, \bf {K} )$, with $m \geq n, n \geq 2 $. For short, let $\PG( m - 1, \bf {K}) = \xi $ and $\PG(m + n - 1, \bf {K}) = \pi $.
\par Next, let $S = \lbrace \alpha_1, \alpha_2, \ldots \rbrace$ be a partition of $\pi \setminus \xi$ by $(n - 1)$-dimensional subspaces of $\pi$. If {\bf K} is the finite field GF($q$), then $\vert S \vert = q^m$.
\par A space $\PG(m + n - 1, \bf{K})$ together with the partition $\lbrace \xi, \alpha_1, \alpha_2, \ldots \rbrace$ will be called a {\em space with an affine spread}.
\par Consider $\alpha_i, \alpha_j \in S$, $i \not= j$, and let $\langle \alpha_i, \alpha_j \rangle \cap \xi = \pi_{ij}$ (here $\langle \alpha_i, \alpha_j \rangle $ is the subspace generated by $\alpha_i$ and $\alpha_j$). Then $\pi_{ij}$ is $(n - 1)$-dimensional. Let $\mathcal {S}_{ij}$ be the Segre variety $\mathcal {S}_{1;n - 1} \subset \langle \alpha_i, \alpha_j \rangle$ determined by $\alpha_i, \alpha_j, \pi_{ij}$ \cite {H&JT:16}.
\bigskip

{\bf Definition}
The set $S$ satisfies the {\em regulus condition} if for any $i, j$, with $ i \not= j$, the regulus \cite {H&JT:16} $\lbrace \pi_{ij}, \alpha_i,  \alpha_j, \ldots \rbrace$ on $\mathcal S_{ij} $, minus $\pi_{ij}$, is a subset $S_{ij}$ of $S$. Notice that this regulus is a set of maximal spaces of the Segre variety $\mathcal S_{ij}$. For short, the regulus condition will also be called {\em condition (R)}. If $S$ satisfies condition (R) it is called a {\em generalized affine space}.

\bigskip
\par If {\bf K} = GF$(q)$, then $\vert S_{ij} \vert = q$. Notice that for $q = 2$  condition (R) is always satisfied. %From now on we assume that $\vert {\bf K} \vert \not= 2$, unless it is stated explicitly that $\vert {\bf K} \vert \geq 2$.
\par Let $\eta_0, \eta_1, \ldots$ be the $m$-spaces in $\pi$ containing $\xi$. Then $\eta_k \cap \mathcal {S}_{ij}$ consists of $\pi_{ij}$ and a line $L^k_{ij}$; this line $L^k_{ij}$ is a maximal line of the Segre variety $\mathcal {S}_{ij}$.

\begin{lemma}
\label{lem2.1}
For fixed $k$ the set of all $L^k_{ij} \setminus \xi$ is the set of all lines of the affine space $\eta_k \setminus \xi$ = $\AG_k(m, \bf K)$. 
\end{lemma} 
{\em Proof} 
Let $x$ and $y$ be distinct points of $\AG_k(m, \bf K)$. Let $\alpha_i$ and $\alpha_j$ be the elements of $S$ containing respectively $x$ and $y$. Then the line $L^k_{ij}$ contains $x$ and $y$. \eop \\

\par Now we define the following bijection $\rho_{kl}$ from $\AG_k(m, \bf K)$ onto $\AG_l(m, \bf K)$, $k \not= l$: 
\par $\rho_{kl}: \AG_k(m, {\bf K}) \rightarrow \AG_l(m, {\bf K}), x^t_k \mapsto x^t_l$,
\par with $\alpha_t \cap \eta_k = \lbrace x^t_k \rbrace$, $\alpha_t \cap \eta_l = \lbrace x^t_l \rbrace$, and $\alpha_t \in S$.

\begin{theorem}
\label{thm2.2}
For $\vert {\bf K} \vert > 2$ the bijection $\rho_{kl}$ is a projectivity from $\AG_k(m, \bf K)$ onto $\AG_l(m, \bf K)$.
\end {theorem}
{\em Proof}
Let $M$ be any line of $\AG_k(m, \bf K)$. Then $M$ is some set $L^k_{ij} \ \setminus \xi$. The line $L^k_{ij}$ belongs to the Segre variety $\mathcal S_{ij}$. Let $L^l_{ij} \cup \pi_{ij}$ be the intersection of $\mathcal S_{ij}$ with the space $\eta_l$. Then $\rho_{kl}$ maps $L^k_{ij} \setminus \xi = M$ onto $L^l_{ij} \setminus \xi = M^\prime$.
\par Hence it is clear that $\rho_{kl}$ is an isomorphism from $\AG_k(m, \bf K)$ onto $\AG_l(m, \bf K)$.
\par By a property of Segre varieties the restriction of $\rho_{kl}$ to $M$ preserves cross-ratios \cite {H&JT:16}.
\par Consequently $\rho_{kl}$ is a projectivity. \eop \\

\begin{corollary}
\label{cor2.3}
For $\vert {\bf K} \vert > 2$ the bijection $\rho_{kl}$ preserves parallelism, that is, parallel lines of $\AG_k(m, \bf K)$ are mapped onto parallel lines  of $\AG_l(m, \bf K)$.
\end{corollary}
\bigskip

\par Let $M_1, M_2, \ldots$ be parallel lines of $\AG_k(m, \bf K)$. The Segre varieties $\mathcal{S}^1_{ij}, \mathcal{S}^2_{{i^\prime}{j^\prime}}, \ldots$ containing the respective lines $M_1, M_2, \ldots$ intersect $\xi$ in spaces $\pi^1_{ij}, \pi^2_{{i^\prime}{j^\prime}}, \ldots$. The lines $M_1, M_2, \ldots$ have a common "point at infinity" $z \in \xi$. Clearly $z$ is a common point of $\pi^1_{ij}, \pi^2_{{i^\prime}{j^\prime}}, \ldots$.

\begin{theorem}
\label{thm2.4}
We have $\pi^1_{ij} = \pi^2_{{i^\prime}{j^\prime}} = \ldots = \bar{\pi}$ and the parallel lines $M_1^\prime, M_2^\prime, \ldots$ of $\AG_l(m, \bf K)$, with $M_i^\prime = M_i^{\rho_{kl}}$ define the same $(n - 1)$-dimensional space $\bar{\pi}$ of $\xi$.
\end{theorem}
{\em Proof}
The parallel lines $M_1, M_2, \ldots$ of $\AG_k(m,\bf K)$ have a common point at infinity $z \in \xi$ which also belongs to $\pi^1_{ij},\pi^2_{{i^\prime}{j^\prime}}, \ldots$. The parallel lines $M_1^\prime = M_1^{\rho_{kl}}, M_2^\prime = M_2^{\rho_{kl}}, \ldots$ of $\AG_l(m, \bf K)$ contain a common point at infinity $z^\prime \in \xi$ which also belongs to $\pi^1_{ij}, \pi^2_{{i^\prime}{j^\prime}}, \ldots$.
\par Varying $l$ it follows that $\pi^1_{ij} = \pi^2_{{i^\prime}{j^\prime}} = \ldots = \bar{\pi}$. \eop \\

\begin{corollary}
\label{cor2.5}
All lines of a parallel class $\Gamma$ of lines in $\AG_k(m, \bf K)$, $\vert \bf K \vert > 2$, determine just one $(n - 1)$-dimensional space $\bar{\pi}$ in $\xi$. This space $\bar{\pi}$ corresponds also to the class  $\Gamma^{\rho_{kl}}$ of parallel lines in $\AG_l(m, \bf K)$, for all $l \not= k$. Hence the set of all spaces $\bar{\pi}$ determined by $\AG_k(m, \bf K)$ is independent of $k$.
\end{corollary}
\bigskip 

\par Let $x \in \AG_k(m, \bf K)$, $\vert \bf K \vert > 2$,  and let $x \in \alpha, \alpha \in S$. Then $x \mapsto \alpha$ defines an $\AG(m, \bf K)$ with point set S and with as lines the sets $S_{ij} \subset S$. This $\AG(m, \bf K)$ is independent of $k$.
\par Let us consider the regulus $\lbrace \pi_{ij}, \alpha_i, \alpha_j, \ldots \rbrace = S_{ij} \cup \lbrace \pi_{ij} \rbrace$. Then by \ref{thm2.4} and \ref{cor2.5} the lines of $\AG(m, \bf K)$ parallel to the line $S_{ij}$ define the same space $\pi_{ij} = \bar \pi$ in $\xi$. Hence parallel lines of $\AG(m, \bf K)$ define the same $(n - 1)$-dimensional space $\bar \pi$ in $\xi$.
\bigskip

\begin{remark}
\label{rem2.6}
\begin{itemize}
\item [(i)] Let $W$ be a subset of $S$. Let $V_k$ be the point set of $\AG_k(m, \bf K)$ consisting of the intersections of $\AG_k(m, \bf K)$ with the elements of $W$ and let $V_l$ be the point set of $\AG_l(m, \bf K)$ consisting of the intersections of $\AG_l(m, \bf K)$ with the elements of $W$.
\par Then $V_k^{\rho_{kl}} = V_l$.
\item [(ii)] It is possible that different parallel classes of lines of $\AG(m, \bf K)$ define the same $(n - 1)$-dimensional space $\bar{\pi}$ in $\xi$. Let $\Pi$ be the set of all these spaces  $\bar{\pi}$.
\item [(iii)] Let $z \in \xi$ and let $L$ be a line of $\AG_k(m, \bf K)$ having $z$ as point at infinity. Then the Segre variety $\mathcal {S}_{ij}$ containing the elements of S intersecting $L$, contains the space $\pi_{ij}$ in $\xi$. Hence $\pi_{ij} \in \Pi$. Consequently $\xi$ is the union of all elements of $\Pi$.
\item [(iv)]Assume that $\Pi$ is an $(n - 1)$-spread of $\xi$ and that ${\bf K}$ = GF($q$), $q > 2$.. Then $\vert \bar \pi \vert$ divides $\vert \xi \vert$, with $\bar \pi \in \Pi$. Hence 
$\frac {q^n - 1} {q - 1}$ divides $\frac {q^m - 1}{q - 1}$
and so $n$ divides $m$.
\end{itemize}
\end{remark}
\bigskip 

{\bf Problem} 
What are the possibilities for the sets $\Pi$ ?
\bigskip

\begin{theorem}
\label{thm2.7}
Let $\PG(m + n - r, \bf K)$ be a subspace of $\PG(m + n - 1, \bf K)$ containing $\xi$, $1 \leq r \leq {n - 1}$. Let $\alpha_i \cap \PG(m + n - r, \bf K) = \beta _i$. Then $\beta_i$ is $(n - r)$-dimensional and $\lbrace \xi, \beta_1, \beta_2, \ldots \rbrace$ is a partition of $\PG(m + n - r, \bf K)$. Also, the set $S_r = \lbrace \beta_1, \beta_2, \ldots \rbrace$ is a generalized affine space.
\end{theorem}
{\em Proof} 
Consider $\beta_i$ and $\beta_j$, $i \not= j$. Let $\mathcal {S}^\alpha_{1; n - 1}$ be the Segre variety defined by $\alpha_i$ and $\alpha_j$, see Section 2.
\par The intersection of $\mathcal {S}^\alpha_{1; n - 1}$ with $\PG(m + n - r, \bf K)$  is a Segre variety $\mathcal {S}^\beta_{1; n - r}$. The variety $\mathcal{S}^\alpha_{1; n - 1}$ contains the regulus $\lbrace \pi_{ij}, \alpha_i, \alpha_j, \ldots \rbrace$ and $\mathcal{S}^\beta_{1; n - r}$ contains the regulus $\lbrace \pi^\prime_{ij}, \beta_i, \beta_j, \ldots \rbrace$, $\pi^\prime_{ij} \subset \pi_{ij}$. Now it is clear that $S_r$ is a generalized affine space. \eop \\

\begin{corollary}
\label{cor2.8}
To prove results on generalized affine spaces, induction is possible.
\end{corollary}

\begin{remark}
\label{rem2.9}
Let $\bf K$ = GF($q$), $q > 2$, and assume that $\Pi$ is a $(n - 1)$-spread of $\xi$. Then $n$ divides $m$, see \ref{rem2.6}. Let $\Pi^\prime$ be the set of $(n - r)$-subspaces of $\xi$ corresponding to $\PG(m + n - r, q)$. If $n - r +1$ does not divide $m$, then $\Pi^\prime$ cannot be a spread of $\xi$.
\end{remark}

\section {the conditions (T$\mathrm i$), $\mathrm i$ = 1, 2}
In this section we define conditions (Ti), i = 1, 2, which will be used to characterize particular generalized affine spaces.
\par Consider again the space $\PG( m + n - 1, {\bf K}), \vert {\bf K} \vert \geq 2$, the space $\xi$ and also the set $S = \lbrace \alpha_1, \alpha_2, \ldots \rbrace$ , with $\xi \subset \PG(m + n - 1, \bf K)$ an $(m - 1)$-dimensional subspace  and $\alpha_i$ an $(n - 1)$-dimensional subspace of $\PG(m + n - 1, \bf K)$. The set $\lbrace \xi, \alpha_1, \alpha_2, \ldots \rbrace$ partitions $\PG(m + n - 1, \bf K)$.
\bigskip

{\bf Definition : The conditions}
\begin{itemize}
\item [(i)] {\bf Condition (T1)} Let $z \in \xi, \langle z, \alpha_i \rangle = \alpha^\prime_i$ with $\alpha_i \in S$, and $\Gamma (z, \alpha_i)$ the set of all elements of $S$ intersecting $\alpha^\prime_i$. If $\alpha_j \in \Gamma (z, \alpha_i)$, assume that $\Gamma (z, \alpha_j) = \Gamma (z, \alpha_i)$, for each choice of $z$ and $\alpha_i$.
\item [(ii)] {\bf Condition (T2)} Assume that $S$ satisfies the regulus condition. Let $\alpha^\prime_i$ be an $n$-dimensional subspace of $\PG(m + n - 1, \bf K)$, $\vert \bf K \vert > 2$, containing $\alpha_i$. Consider all elements $\alpha_{j_1}, \alpha_{j_2}, \ldots$ of $S$ containing a point of $\alpha_i^\prime$ (then $\lbrace \alpha_{j_1}, \alpha_{j_2}, \ldots  \rbrace = \Gamma (z, \alpha_i)$, with $\lbrace z \rbrace = \xi \cap \alpha^\prime_i$). Assume that for each pair $(\alpha_i, \alpha^\prime_i), \alpha_i \in S$, the set $\lbrace \alpha_{j_1},\alpha_{j_2}, \ldots \rbrace$ is an affine subspace of $\AG(m, \bf K)$.
\end{itemize}
\bigskip

\begin{theorem}
\label{thm3.1}
For generalized affine spaces with $\vert {\bf K} \vert > 2$ conditions (T1) and (T2) are equivalent.
\end{theorem}
{\em Proof}
Let $S$ be a generalized affine space, with $S = \lbrace \alpha_1, \alpha_2, \ldots \rbrace$.
\par Assume condition (T1) is satisfied. We have to prove that $\Gamma (z, \alpha_i)$, with $z \in \xi$ and $\alpha_i \in S$, is an affine subspace of $\PG(m, \bf K)$. Let $\alpha_{j_1}, \alpha_{j_2} \in \Gamma (z, \alpha_i), j_1 \not= j_2$. Then $\alpha_{j_2} \in \Gamma (z, \alpha_{j_1})$. Let $L$ be the line containing $z$ and intersecting $\alpha_{j_1}$ and $\alpha_{j_2}$
The elements of S containing a point of $L \setminus \xi$ are contained in $\Gamma (z, \alpha_i)$. Hence the line of $\AG(m, \bf K)$ containing
$\alpha_{j_1}, \alpha_{j_2}$ belongs to $\Gamma (z, \alpha_i)$, and so, as $\vert {\bf K} \vert > 2$, $\Gamma (z, \alpha_i)$ is a subspace of $\AG(m, \bf K)$ (p. 24 of \cite {U:11}). Hence (T2) is satisfied.
\par With each line $zy$ with $y \in \alpha_i$ corresponds a parallel class of $\Gamma (z, \alpha_i)$, and so the subspace $\Gamma (z, \alpha_i)$ of $\AG(m, \bf K)$ is $n$-dimensional for $\bf K$ = GF($q$).
\par Next assume that condition (T2) is satisfied. Consider $z \in \xi, \langle z, \alpha_i \rangle = \alpha_i^\prime$ with $\alpha_i \in S$, and also the set $\Gamma (z, \alpha_i) = \AG(r, \bf K)$. Let $\alpha_j \in \Gamma (z, \alpha_i), i \not= j$. By {\ref {thm2.4}} the set of spaces $\langle \alpha_i, \alpha_u \rangle \cap \xi = \pi^\ast$ with $\alpha_u \in \AG(r, \bf K)$, $u \not= i$, and the set of spaces $\langle \alpha_j, \alpha_u \rangle \cap \xi = \pi^{\ast \ast}$ with $\alpha_u \in \AG(r, \bf K)$, $u \not= j$, coincide. All these spaces $\pi^\ast$ and $\pi^{\ast\ast}$ contain the point $z$. It follows that $\alpha_u \in \Gamma (z, \alpha_j)$. Consequently $\Gamma (z, \alpha_i) = \Gamma (z, \alpha_j)$.  \eop \\
\bigskip

\begin{remark}
\label{rem3.2}
\begin{itemize}
\item [(i)] Let us consider a space $\PG(m + n - 1, {\bf K}), \vert {\bf K} \vert \geq 2$, with an affine spread $\lbrace \xi, \alpha_1, \alpha_2, \ldots \rbrace$. Let $S = \lbrace \alpha_1, \alpha_2, \ldots \rbrace, z \in \xi, \alpha_i \in S, \Gamma (z, \alpha_i) = S^\prime$ and assume that (T1) is satisfied. Then for any two elements $\alpha_j, \alpha_u \in S^\prime, j \not= u$, there is exactly one line $L$ on $z$ intersecting $\alpha_j, \alpha_u$ and this line $L$ intersects exactly $\vert \bf K \vert$ elements of $S^\prime$.  
\par Conversely, consider in $\PG(m + n - 1, \bf K)$, $\vert \bf K \vert \geq 2$, a generalized affine space with affine spread $\lbrace \xi, \alpha_1, \alpha_2, \cdots \rbrace$. If for any $z \in \xi$ and any two spaces $\alpha_i, \alpha_j, i \not= j$, there is exactly one iine  $L$ containing $z$ and intersecting  $\alpha_i$ and $\alpha_j$, then (T1) is satisfied.
\item [(ii)] Consider a generalized affine space $S$ in the space $\PG(m + n - 1, {\bf K}), \vert {\bf K} \vert > 2$, and assume that (T1) is satisfied. Then, for fixed $z \in \xi$, the affine subspaces $\Gamma (z, \alpha_i), \alpha_i \in S$, partition $\AG(m, \bf K)$.
\item [(iii)] Consider again a space $\PG(m + n - 1, {\bf K}), \vert {\bf K} \vert \geq 2$, with an affine spread $\lbrace \xi, \alpha_1, \alpha_2, \ldots \rbrace$. Let $z \in \xi, \alpha_i \in S$, and assume that (T1) is satisfied. Then, for fixed $z \in \xi$ the sets $\Gamma (z, \alpha_i), \alpha_i \in S$, partition $S$.
\end{itemize}
\end{remark}
\bigskip

\section {Projections}
Choose a hyperplane $\PG(m + n - 2, \bf K)$ in $\PG(m + n - 1, \bf K), \xi \not \subseteq \PG(m + n - 2, \bf K)$, and choose a point $z \in \xi, z \not \in \PG(m + n - 2, \bf K)$. We project from $z$ onto $\PG(m + n - 2, \bf K)$
and consider the projection of $S^\prime = \Gamma (z, \alpha_i)$ with $\alpha_i \in S$. We assume that (T1) is satisfied and that $\vert {\bf K} \vert \geq 2$.
\par Let $\alpha^\ast_j$ be the projection of $\alpha_j \in S^\prime$ and let $S^\ast$ be the set of all $\alpha^\ast_j$. The union of all elements of $S^\ast$ is $\widetilde{S}$. The projection of $\xi \setminus \lbrace z \rbrace$ is an $(m - 2)$-dimensional space $\xi^\ast$.
\par From Section 3 follows that each point of $\widetilde{S}$ is contained in $\vert \bf K \vert$ elements of $S^\ast$ and that any two distinct elements of $S^\ast$ have exactly one point in common. If $\alpha_i$ varies in $S$, then all sets $\widetilde{S}$ form a partition $\mathcal {P}$ of $\PG(m + n -2, \bf K) \setminus \xi^\ast$.
\par Now let $\bf K$ = GF($q$). Then $\vert \widetilde{S} \vert = q^{n - 1}(q^n - 1)/(q - 1)$, each point of $\widetilde{S}$ is contained in $q$ elements of $S^\ast$ and any two distinct elements of $S^\ast$ have exactly one point in common. If $\alpha_i$ varies in the set $S$, then the $q^{m - n}$ sets $\widetilde S$ form a partition $\mathcal {P}$ of $\PG(m + n - 2, q) \setminus \xi^\ast$.
\par Particularly interesting is the case $q = 2$, see Section 8. Then $ \vert \widetilde S \vert = 2^{n - 1}(2^n - 1)$, each point of $\widetilde S$ is contained in two elements of $S^\ast$ and any two distinct elements of $S^\ast$ have exactly one point in common. If $\alpha_i$ varies in $S$, then the $2^{m - n}$ sets $\widetilde S$ form a partition $\mathcal P$ of $\PG(m + n - 2, 2) \setminus \xi^\ast$.

\section {Generalized affine spaces satisfying (T1)}
We consider a generalized affine space $S$ in $\PG(m + n - 1, \bf K), \vert \bf K \vert > 2$,satisfying condition (T1), and consequently also condition (T2).
\bigskip

\begin{theorem}
\label{thm5.1}
Let $\Pi$ consist of all ''spaces at infinity" of the generalized affine space $\AG(m, \bf K)$ defined by the set $S$ which satisfies (T1). Then $\Pi$ is an $(n - 1)$-spread of $\xi$. Consequently, for $\bf K$ = GF($q$) the integer $n$ divides $m$.
\end{theorem}
{\em Proof}
Let $z \in \xi$ and $\alpha_i \in S$. Let $\alpha_u, \alpha_v$ be distinct elements of the set $\Gamma (z, \alpha_i) = \AG(r, \bf K)$, see Section 3. Then $\alpha_v \in \Gamma (z, \alpha_u)$ by (T1). So there is a line $L$ containing $z$ and intersecting $\alpha_u$ and $\alpha_v$. Consider the $\vert {\PG(r - 1, \bf K)} \vert$ lines of $\AG(r, \bf K)$ containing $\alpha_i$. These lines represent the $\vert {\PG( r - 1, \bf K)} \vert$ parallel classes of lines in $\AG(r, \bf K)$.
\par These parallel classes define {\bf at most} $\vert \PG(r - 1, \bf K) \vert$ {\bf distinct} $(n - 1)$-spaces in $\xi$, and they all belong to $\Pi$. Let $\pi_{j_1}, \pi_{j_2}, \ldots$ be these distinct $(n - 1)$-spaces . The point $z$ belongs to all these $(n - 1)$-spaces. Also the point $z^{\rho_{kl}}$ belongs to all these spaces, see Section 2. It follows that $\pi_{j_1} = \pi_{j_2} = \ldots = \pi_{z, \alpha_i}$; remark that $\pi_{z, \alpha_i} = \pi_{z, \alpha_u}$ for all $\alpha_u \in \Gamma (z, \alpha_i)$.
\par Consequently all elements of $\Gamma (z, \alpha_u)$ are contained in $\langle \pi_{z, \alpha_i}, \alpha_i \rangle$ for all $\alpha_u \in \Gamma (z, \alpha_i)$.
\par Let $z^\prime \in \pi_{z, \alpha_i}$. As $\langle \pi_{z, \alpha_i}, \alpha_i \rangle$ is $(2n - 1)$-dimensional each element of $\Gamma (z, \alpha_i)$ contains a point of $\langle z^\prime, \alpha_i \rangle$. Hence $\Gamma (z, \alpha_i) \subseteq \Gamma (z^\prime, \alpha_i)$. So $\pi_{z^\prime, \alpha_i} = \pi_{z, \alpha_i}$ and $\Gamma (z^\prime, \alpha_i) \subseteq \Gamma (z, \alpha_i)$. That is, $\Gamma (z^\prime, \alpha_i)$ is independent of the choice of $z^\prime$ in $\pi_{z, \alpha_i}$.
\par The spaces of $\Gamma (z, \alpha_i)$ are disjoint $(n - 1)$-spaces in $\langle \pi_{z, \alpha_i}, \alpha_i \rangle$. Assume, by way of contradiction, that there is a point $y$ in $\langle \pi_{z, \alpha_i}, \alpha_i \rangle$ which is not in an element of $\Gamma (z, \alpha_i)$. Let $y \in \alpha_s, s \not= i$, and let $\lbrace z^\prime \rbrace = \pi_{z, \alpha_i}  \cap \langle y, \alpha_i \rangle$. Then $\alpha_s \in \Gamma (z^\prime, \alpha_i)$. Hence $\alpha_s \in \Gamma (z, \alpha_i)$ and so $\alpha_s \subset \langle \pi_{z, \alpha_i}, \alpha_i \rangle$, a contradiction. It follows that $\Gamma (z, \alpha_i) \cup \lbrace \pi_{z, \alpha_i} \rbrace$ is an $(n - 1)$-spread of $\langle \pi_{z, \alpha_i}, \alpha_i \rangle$.
\par Now we show that the spaces $\pi_{z, \alpha_i}$, $z$ any point of $\xi$ and $\alpha_i$ any element of $S$, define an $(n - 1)$-spread of $\xi$.
\par Consider $\pi_{z, \alpha_i}$ and $\pi_{z^\prime, \alpha^\prime_i}$ with $z, z^\prime \in \xi$ and $\alpha_i, \alpha^\prime_i \in S$, and assume that $z^{\prime\prime} \in \pi_{z, \alpha_i} \cap \pi_{z^\prime, \alpha^\prime_i}$. Then $\Gamma (z, \alpha_i) = \Gamma (z^{\prime\prime}, \alpha_i)$ and $\Gamma (z^\prime, \alpha^\prime_i) = \Gamma  (z^{\prime\prime}, \alpha^\prime_i)$. Let $\eta_k$ be an $m$-dimensional subspace of $\PG(m + n - 1, \bf K)$ containing $\xi$. Then $\langle \pi_{z, \alpha_i}, \alpha_i \rangle \cap \eta_k = \rho_k$ is $n$-dimensional and $\rho_k = \langle \pi_{z, \alpha_i}, u \rangle$ with $u \in \alpha_i$. The affine space $\AG_k(m, \bf K)$ defined by $\eta_k$ contains the affine line $uz^{\prime\prime} \setminus \lbrace z^{\prime\prime} \rbrace = L$, and, similarly, the affine space $\AG_k(m, \bf K)$ contains an affine line $u^{\prime} z^{\prime\prime} \setminus \lbrace z^{\prime\prime} \rbrace = L^\prime$ with $u^\prime \in \alpha^{\prime}_i$. In $\AG_k(m, \bf K)$ the lines $L$ and $L^\prime$ are parallel. Consequently $\pi_{{z^{\prime\prime}}, \alpha_i} = \pi_{{z^{\prime\prime}}, \alpha^\prime_i}$.
\par We conclude that the "spaces at infinity"  of $S$ define an $(n - 1)$-spread of $\xi$.
\par Hence, if $\bf K$ = GF($q$), then by \ref {rem2.6}  the integer $n$ divides $m$. \eop \\
\bigskip

\par Assume again that $S$ satisfies (T1), with $\vert \bf K \vert > 2$, and let $\Pi$ be the set of all spaces at infinity $\pi_{z, \alpha_i}$. Let $B$ be the set of all spaces $\langle \pi_{z, \alpha_i}, \alpha_i \rangle, z \in \xi$ and $\alpha_i \in S$, and let us consider the "point-line" incidence structure $\mathcal D = (S, B, \in)$.
\par By the proof of {\ref{thm5.1}} $\Gamma (z, \alpha_i) \cup \lbrace \pi_{z, \alpha_i} \rbrace$ is an $(n - 1)$-spread of $\langle \pi_{z, \alpha_i}, \alpha_i \rangle$.
\par If $\beta, \beta^\prime \in B$ define the same space at infinity $\bar {\pi} \in \Pi$, then we say that $\beta$ and $\beta^\prime$ are {\em parallel} and write $\beta \parallel \beta^\prime$.
\bigskip

\begin{theorem}
\label{thm5.2}
The incidence structure $\mathcal D$ is an affine space $\AG(m/n, \bf L)$ over a skewfield $\bf L$, with $\bf L$ an $n$-th extension of $\bf K$. Hence $n$ divides $m$.
\end{theorem}
{\em Proof}
First of all any two distinct elements of $S$      are contained in exactly one element of $B$.
\par Let $\beta \in B$ and let $\bar {\pi} \in \Pi$ be the space at infinity of $\beta$. If $\alpha \in S$, then $\langle \bar {\pi}, \alpha \rangle = \beta^\prime$ is the unique element of B containing $\alpha$ and paralllel to $\beta$.
\par For $\mathcal D$ to be an affine space it is sufficient that any three non-collinear points of $\mathcal D$ (that are three elements $\alpha_1, \alpha_2, \alpha_3$ of $S$ not belonging to a common element of $B$) generate an affine subplane of the geometry $\mathcal D$, see page 24 of  \cite {U:11}. Remark that the geometry $\mathcal A$ generated by $\alpha_1, \alpha_2, \alpha_3$ is contained in the $(3n - 1)$-dimensional subspace $\langle \alpha_1, \alpha_2, \alpha_3 \rangle$ of $\PG(m + n -1, \bf K)$. Also, distinct lines $\beta, \beta^\prime$ of $B$ in $\mathcal A$ are $(2n - 1)$-subspaces in the $(3n - 1)$-subspace $\langle \alpha_1, \alpha_2, \alpha_3 \rangle$, hence have an $(n - 1)$-subspace in common. As $S \cup \Pi$ induces an $(n - 1)$-spread in each element of $B$, the lines $\beta, \beta^\prime$ are either parallel or have an element of $S$ in common. Remark also that for any line $\beta$ in $\mathcal A$ and any point $\alpha \in \mathcal A$ there is exactly one line in $\mathcal A$ containing $\alpha$ and parallel to $\beta$.
\par Now it is clear that $\mathcal A$ is an affine subplane of $\mathcal D$ and hence $\mathcal D$ is an affine space.
\par By considering a maximal chain of subspaces of $\mathcal D$, it follows that the dimension d of $\mathcal D$ is determined by the equation $m + n - 1 = (d + 1)n - 1$. Hence $d = m/n$ and so $n$ divides $m$.
\par Let $\mathcal D^\prime$ be the projective completion of $\mathcal D$. The point set of $\mathcal D^\prime$ is the $(n - 1)$-spread $S \cup \Pi$ of $\PG(m + n - 1, \bf K)$. In the terminology of Segre, see IX of (\cite {S:00}, \cite {S:64}), the spread $S \cup \Pi = P$ is a "{\bf sistema grafico}" and so $\mathcal D^\prime$ is a $\PG(m/n, \bf L)$ over some $n$-th skewfield extension of $\bf K$. \eop \\
\bigskip

\begin{corollary}
\label{cor5.3}
\begin{itemize} 
\item [(i)] The set $\Pi$ is a $\PG(m/n - 1, \bf L)$ in $\mathcal D^\prime$. Let ${\overline{\PG}}(m/n - 1, \bf L)$ be any hyperplane in $\mathcal D^\prime$ and let $\overline \Pi$ be the set of all elements of $S \cup \Pi$ in ${\overline{\PG}}(m/n - 1, \bf L)$. Then ${\overline S} = (S \cup \Pi) \setminus {\overline \Pi}$ is a generalized affine space satsifying (T1), see page 1083 of \cite {S:00} or page 47 of \cite {S:64}. For condition (R) see page 1083 of \cite {S:00} or page 47 of \cite {S:64}. Condition (T1) follows from the fact that $\Gamma (z, \alpha_i) \cup \lbrace \pi_{z, \alpha_i} \rbrace$, with $z \in \xi$ and $\alpha_i \in S$, is an $(n - 1)$-spread of the $(2n - 1)$-dimensional space $\langle \alpha_i , \pi_{z, \alpha_i} \rangle \in B$ and that the elements of $\Pi$ in the "lines" of $\mathcal D^\prime$ not in $\mathcal D$ also form $(n - 1)$-spreads.
\item [(ii)] Let $\overline {\alpha}_i, \overline {\alpha}_j, \overline {\alpha}_k$ be distinct elements of $S \cup \Pi$ contained in some $(2n - 1)$-space of $\PG(m + n - 1, \bf K)$. Then the regulus defined by $\overline {\alpha}_i, \overline {\alpha}_j, \overline {\alpha}_k$ is contained in $S \cup \Pi$, see page 1083 of \cite{S:00} or page 47 of \cite{S:64}. Notice that $\overline {\alpha}_i, \overline {\alpha}_j, \overline {\alpha}_k$ are three distinct collinear points of $\PG (m/n, \bf L)$.
\end{itemize}
\end{corollary}
\bigskip

\begin{remark}
\label{rem5.4}
In IX of (\cite {S:00}, \cite {S:64}) Segre describes how to construct $S$ from the skewfield extension $\bf L$ of the field $\bf K$. He also shows that from {\bf {any}} $n$-th skewfield extension $\bf L$ of some field $\bf K$ ($\bf K \not= \bf L$ and $\vert \bf K \vert \geq 2$) a generalized affine space $S$ satisfying (T1) can be constructed.
\end{remark}
\bigskip

{\bf Definition}
Let the field $\bf L$ be an $n$-th algebraic extension of the field $\bf K$ ($n \geq 2$ and $\vert \bf K \vert \geq 2$). In the corresponding extension $\PG (m + n - 1, \bf L)$ of $\PG (m + n - 1, \bf K)$, $m = dn$ with $d \geq 2$, we consider $n$ subspaces $\mathcal P_1, \mathcal P_2, \ldots, \mathcal P_n$ of dimension $d$ over $\bf L$ which span $\PG (m + n - 1, \bf L)$ and are conjugate with respect to the extension $\bf L$ of $\bf K$. Now consider all $(n - 1)$-spaces $\alpha$ over $\bf K$ which, extended to $\bf L$, intersect $\mathcal P_1, \mathcal P_2, \ldots, \mathcal P_n$. Let $S^\ast$ be the set of all these spaces $\alpha$. Let $\mathcal L_1$ be a hyperplane of $\mathcal P_1$ and let $\mathcal L_2, \mathcal L_3 \ldots, \mathcal L_n$ be its conjugates in $\mathcal P_2, \mathcal P_3, \ldots, \mathcal P_n$. Let $\Pi$ be the set of all $(n - 1)$-spaces over $\bf K$ which correspond to the points of $\mathcal L_1$ and its conjugates. Then $S^\ast \setminus \Pi = S$ is a generalized affine space (satisfying (T1)).
\par Alternatively, let $V$ be the vector space over $\bf L$ underlying the projective space $\PG(d, {\bf L}), d \geq 2$. If $V$ is considered as a $\bf K$-vector space, each point of $\PG(d, \bf L)$ becomes an $(n - 1)$-dimensional subspace of $\PG(dn + n - 1, \bf K)$. If $\AG(d, \bf L)$ is an affine space arising from $\PG(d, \bf L)$, then it becomes a generalized affine space in $\PG(dn + n - 1, \bf K)$ (satisfying (T1)).
\par Such generalized affine spaces and the corresponding spreads $S \cup \Pi$ are called {\em regular}, "{\em elementari}" in the terminology of Segre, see VII of (\cite {S:00}, \cite {S:64}).
\bigskip 

\begin{theorem}
\label{thm5.6}
For $\bf K$ = GF($q$), $q > 2$, all generalized affine spaces satisfying (T1) (or, equivalently (T2)) are regular.
\end{theorem}
{\em Proof}
The spread $S \cup \Pi$ of $\PG(m + n - 1, \bf K)$ is a "sistema grafico". By Segre (\cite {S:00}, \cite {S:64}) each finite sistema grafico is regular; see page 1086 of \cite {S:00} or page 50 of \cite {S:64}. \eop \\
\bigskip 

\begin{remark}
\label{rem5.6}
Consider a space $\PG(m + n - 1, \bf K)$, $\vert \bf K \vert > 2$, with an affine spread $\lbrace \xi, \alpha_1, \alpha_2, \ldots \rbrace$ and assume that for $S = \lbrace \alpha_1, \alpha_2, \dots \rbrace$ conditions (R) and (T1) are satisfied. Hence $S$ is a generalized affine space. Then by Section 3 $\Gamma (z, \alpha_i)$, $z \in \xi$, is an affine subspace of $\AG(m, \bf K)$. By \ref {thm5.1} $\Gamma (z, \alpha_i) \cup \lbrace \pi_{z, \alpha_i}  \rbrace$ is an $(n - 1)$-spread of $\langle \alpha_i, \pi_{z, \alpha_i} \rangle$. Also, $\PG_k(m, \bf K) \cap \langle \alpha_i, \pi_{z, \alpha_i} \rangle$ is an $n$-space containing $\pi_{z, \alpha_i}$. Hence the intersections of $\AG_k(m, \bf K)$ with the elements of $\Gamma (z, \alpha_i)$ are the points of an $\AG_k(n, {\bf K}) \subset \AG_k(m, {\bf K})$. Consequently the elements of $\Gamma (z, \alpha_i)$ form an $\AG(n, \bf K)$ in $\AG(m, \bf K)$. We conclude that in condition (T2) and in the proof of \ref {thm3.1} the affine subspaces of $\AG(m, \bf K)$ are $n$-dimensional.
\end{remark}

\section {Generalized ovoids}
Here we apply our results to generalized ovoids of $\PG(4n - 1, q), q \geq 2$ and $n \geq 1$. Most of it can be generalized to the infinite case.
\bigskip

{\bf Definition}
In $\Omega = \PG(4n - 1, q)$, let $O$ be a set of $(n - 1)$-dimensional subspaces $\xi_i, i = 0, 1, \ldots, q^{2n}$, such that
\begin {itemize}
\item [(a)] every three generate a $\PG(3n - 1, q)$;
\item [(b)] for every $i \in \lbrace 0, 1, \ldots, q^{2n} \rbrace$, there is a $(3n - 1)$-dimensional subspace $\tau_i$ that contains $\xi_i$ and is disjoint from $\xi_j$ for $j \not= i$.
The space $\tau_j$ is the {\em tangent space} of $O$ at $\xi_j$; it is uniquely defined by $O$ and $\xi_j$.
\end {itemize}
The set $O$ is a {\em generalized ovoid} or a {\em pseudo-ovoid} or an {\em egg} or an {\em [n - 1]-ovoid} of $\PG(4n - 1, q)$ \cite {TTVM:06}.
 
\bigskip

\begin {remark}
\label {rem6.1}
A $[0]$-ovoid of $\PG(3, q)$ is just an ovoid of $\PG(3, q)$.
\end {remark}
\bigskip

{\bf Definition}
In the extension $\PG(4n - 1, q^n)$ of $\PG(4n - 1, q)$ consider $n$ 3-spaces $\zeta_i, i = 1, 2, \ldots, n$, that are conjugate in the extension GF($q^n$) of GF($q$) and which span $\PG(4n - 1, q^{n})$. This means that they form an orbit of the Galois group corresponding to this extension and span $\PG(4n - 1, q^n)$.
\par In the space $\zeta_1$ take an ovoid $O_1 = \lbrace x^{(1)}_0, x^{(1)}_1, \ldots, x^{(1)}_{q^{2n} }\rbrace$. Next, let $x^{(1)}_i, x^{(2)}_i, \ldots, x^{(n)}_i$, $i = 0, 1, \ldots, q^{2n}$, be conjugate in GF($q^n$) over GF($q$). The points $x^{(1)}_i, x^{(2)}_i, \ldots, x^{(n)}_i$ now define an $(n - 1)$-dimensional space $\xi_i$ over GF($q$) for each $i = 0, 1, \ldots, q^{2n}$. It follows that the set $O = \lbrace \xi_0, \xi_1, \dots, \xi_{q^{2n}} \rbrace$ is a pseudo-ovoid of $\PG(4n - 1, q)$.
\par These are the {\em regular}  or {\em elementary} pseudo-ovoids. If $O_1$ is an elliptic quadric over GF($q^n$), the corresponding pseudo-ovoid is {\em classical} or a {\em pseudo-quadric}.
\par Alternatively, let $V$ be the 4-dimensional vector space underlying the projective space $\PG(3, q^n)$. Considering $V$ as a GF($q$)-vector space, each point of $\PG(3, q^n)$ becomes an $(n - 1)$-dimensional subspace of $\PG(4n - 1, q)$. If $O_1$ is an ovoid of $\PG(3, q^n)$, then $O_1$ becomes a regular pseudo-ovoid of $\PG(4n - 1, q)$.
\bigskip

\begin{remark}
\label{rem6.2}
For $q$ even, every known generalized ovoid is regular \cite {TTVM:06}. For $q$ odd, there are generalized ovoids that are not regular \cite {TTVM:06}. By the theorem of Barlotti \cite {B:55} and Panella \cite {P:55}, for $q$ odd, every regular generalized ovoid is a pseudo-quadric.
\end{remark}
\bigskip

{\bf  Definition}
\begin{itemize}
\item [(i)] The pseudo-ovoid $O$ in $\PG(4n - 1, q)$ is good at its element $\xi_i$ if any $\PG(3n - 1, q)$ containing $\xi_i$ and at least two other elements of $O$ contains exactly $q^n + 1$ elements of $O$.
\item [(ii)] In this case, $\xi$ is a {\em good} element of $O$, and $O$ is also said to be {\em good (at $\xi_i$)},
\end{itemize}
\bigskip

\begin{remark}
\label{rem6.3}
\begin {itemize}
\item [(i)] A regular pseudo-ovoid is good at each of its elements.
\item [(ii)] There is a large amount of literature on generalized ovoids and on good generalized ovoids \cite {TTVM:06}.
\end {itemize}
\end {remark}
\bigskip

\par Consider a generalized ovoid $O = \lbrace \xi_0, \xi_1, \ldots, \xi_{q^{2n}} \rbrace$ in the space $\PG(4n - 1, q), n \not= 1$. Let $\PG(3n - 1, q)$ be a $(3n - 1)$-dimensional subspace of $\PG(4n - 1, q)$ which is skew to $\xi_0$. Now we project from $\xi_0$ onto $\PG(3n - 1, q)$.
\par Let the projection of $\xi_i, i \not= 0$, be $\alpha_i$ and let the projection of $\tau_0 \setminus \xi_0$ be $\xi$. Then $\alpha_i$ is $(n - 1)$-dimensional and $\xi$ is $(2n - 1)$-dimensional. Also, $\PG(4n - 1, q)$ together with the partition $\xi \cup S$, with $S = \lbrace \alpha_1, \alpha_2, \ldots, \alpha_{q^{2n}} \rbrace$ is a space with an affine spread.
\bigskip
\begin{theorem}
\label{thm6.4}
If $q \not= 2$ and $S$ is a generalized affine space satisfying condition (T1), or equivalently (T2), then the generalized ovoid $O$ is good with good element $\xi_0$.
\end{theorem}
{\em Proof}
This follows from \ref {thm5.1} and \ref {thm5.2}; crucial is $\Pi$ being an $(n -1)$-spread of $\xi$. \eop \\
\bigskip

\begin{corollary}
\label{cor6.5}
Let $q \not=  2$. Assume that, projecting consecutively from $\xi_0, \xi_1, \ldots, \xi_{q^{2n}}$, the resulting set $S$ always is a generalized affine space satisfying (T1), or equivalently (T2). Then the generalized ovoid $O$ is good at each element. Also, for $q$ odd the generalized ovoid $O$ is classical and for $q$ even it is regular.
\end{corollary}
{\em Proof}
This follows from \ref {thm6.4} and 5.1.12 of \cite {TTVM:06}. \eop \\
\bigskip

\begin{remark}
\label{rem6.6}
\begin{itemize}
\item [(i)] For a strenghtening of \ref {cor6.5}, see 5.1.12 of \cite {TTVM:06}.
\item [(ii)] Assume that for $S$ condition (T1) is satisfied. For $O$ this means the following.
\par Let $\gamma$ be an $n$-dimensional subspace of $\tau_0$ containing $\xi_0$. Consider the space $\langle \gamma, \xi_i  \rangle = \delta_i, \xi_i \in O \setminus \lbrace \xi_o \rbrace$. Then $\delta_i$ intersects exactly $q^n + 1$ elements $\xi_0, \xi_i, \xi_{j_1}, \ldots, \xi_{j_{q^n - 1}}$ of $O$. Notice that $\langle \gamma, \xi_i \rangle \cap \xi_{j_k}, k = 1, 2, \ldots, q^n - 1$, is a point. Let $\lbrace \xi_0, \xi_i, \xi_{j_1}, \ldots, \xi_{j_{q^n - 1}} \rbrace = O^\ast$. If (T1) is satisfied, then $\langle \gamma, \xi_{j_k} \rangle = \delta_{j_k} , k = 1, 2, \ldots, q^n -1$, also intersects all elements of $O^\ast$. Hence $O^\ast$ is uniquely defied by $\gamma$ and any of the elements of $O^\ast \setminus \lbrace \xi_0 \rbrace$.
\par Also the converse holds.
\end{itemize}
\end{remark}
\bigskip

\par Next, a new condition (E) will be introduced.
\bigskip

{\bf Definition}
Let $O = \lbrace \xi_0, \xi_1, \ldots, \xi_{q^{2n}} \rbrace$ be a generalized ovoid in $\PG(4n - 1, q), q \geq 2$. Let $\xi_i, \xi_j, \xi_k$ be three distinct elements of $O$, let $\langle \xi_i, \xi_j, \xi_k \rangle$ be $\Delta$, and let $\rho$ be the $(n - 1)$-dimensional space $\tau_i \cap \tau_j \cap \tau_k$. We say that $O$ satisfies {\em condition (E)} with respect to $\lbrace \xi_i, \xi_j, \xi_k \rbrace$ if $\rho$ is contained in $\Delta$.
\bigskip

\begin{remark}
\label{rem6.7}
If the generalized ovoid $O$ is regular and $q$ is even, then $O$ satisfies (E) for each triple $\lbrace \xi_i, \xi_j, \xi_k \rbrace \subset O$ (\cite {H:85}, \cite {TTVM:06}).
\end{remark}
\bigskip

\begin{theorem}
\label{thm6.8}
Assume that $O \subset \PG(4n - 1, q)$ satisfies (E) for each triple $\lbrace \xi_i, \xi_j, \xi_k \rbrace \subset O$, where $\xi_i$ is fixed. If we project from $\xi_i$ onto a $\PG(3n - 1, q) \subset \PG(4n - 1, q)$ skew to $\xi_i$, then the resulting set $S$ satisfies condition (T1).
\end{theorem}
{\em Proof}
Let $x$ be a point of $\rho = \tau_i \cap \tau_j \cap \tau_k$. The point $x$ is contained in $q^n + 1$ tangent spaces $\tau_i, \tau_j, \tau_k, \ldots$ of $O$ \cite {TTVM:06}. Let $W$ be the set of these tangent spaces and let $U$ be the set of the elements $\lbrace \xi_i, \xi_j, \xi_k, \ldots \rbrace$ of $O$ contained in these spaces. Then for $\xi_u \in U \setminus \lbrace \xi_i, \xi_j \rbrace$ the space $\langle \xi_i, \xi_j, x \rangle$ is $2n$-dimensional and the space $\langle \xi_i, \xi_j, \xi_u \rangle$ containing $x$ is $(3n - 1)$-dimensional, so $\langle \xi_i, \xi_j, x \rangle$ contains  a point $y$ of $\xi_u$. Hence  $\langle \xi_i, \xi_j, x \rangle$ intersects the $q^n + 1$ spaces of $U$. Interchanging roles of $\xi_j$ and $\xi_u$, we see that $\langle \xi_i, \xi_u, x \rangle$ intersects all spaces of $U$.
\par From \ref {rem6.6} it follows that condition (T1) is satisfied for $S$. \eop \\
\bigskip

\begin{corollary}
\label{cor6.9}
\begin{itemize}
\item [(i)] Let $q \not= 2$. If $O$ satisfies (E) for each triple $\lbrace \xi_i, \xi_j, \xi_k \rbrace \subset O$, with $\xi_i$ fixed, and if $S$ satisfies (R), then $O$ is good at $\xi_i$. As $O$ is good at $\xi_i$ and satisfies (E) it follows from \cite {TTVM:06} that $q$ is even.
\item [(ii)] Let $q \not= 2$. If for a given $\xi_i \in O$ $S$ is a generalized affine space which satisfies (T1) then $O$ is good at $\xi_i$, so by \cite {TTVM:06} $O$ satisfies (E) for each triple $\lbrace \xi_i, \xi_j, \xi_k \rbrace \subset O$ if and only if $q$ is even.
\item [(iii)] Let $q \not= 2$. If $O$ satisfies (E) for each triple $\lbrace \xi_i, \xi_j, \xi_k \rbrace \subset O$ and if for each $\xi_i \in O$ the set $S$ satisfies (R), then $O$ is good at each element of $O$. Hence $q$ is even and by \ref {cor6.5} the generalized ovoid $O$ is regular.
\end{itemize}
\end{corollary}
\bigskip

\section {Generalized quadrangles}
Here our results will be applied to finite translation generalized quadrangles (TGQ) of order $(s, s^2), s \not= 1$. Again much of it can be generalized to the infinite case.
\bigskip

{\bf Definitions}
A (finite) {\em generalized quadrangle (GQ)} is an incidence structure $\mathcal {S} = (P, B, \bf I)$ in which $P$ and $B$ are disjoint (nonempty) sets of objects called {\em points} and {\em lines}  (respectively) and for which {\bf I} is a symmetric point-line incidence relation satisfying the following axioms:
\begin {itemize}
\item [(i)] Each point is incident with $1 + t$ lines $(t \geq 1)$ and two  distinct points are incident with at most one line.
\item [(ii)] Each line  is incident with $1 + s$ points $(s \geq 1 )$ and two distinct lines are incident with at most one point.
\item [(iii)] If $x$ is a point and $L$ is a line not incident with $x$, then there is a unique pair $(y, M) \in P \times B$ for which $x \mathbf {I} M \mathbf {I} y \mathbf {I} L$.
\end {itemize}
\par The integers $s$ and $t$ are the {\em parameters} of the GQ and $\mathcal S$ is said to have {\em order} $(s, t)$; if $s = t$, then $\mathcal S$ is said to have {\em order} $s$.
\bigskip

\par For terminology, notations, results, etc., concerning finite generalized quadrangles and not explicitly given here, see the monograph of Payne and Thas \cite {P&JT:09}.
\par Let $O = \lbrace \xi_0, \xi_1, \ldots, \xi_{q^{2n}} \rbrace, n \geq 1$, be a generalized ovoid in $\PG(4n - 1, q)$. Let $\tau_i$ be the tangent space of $O$ at $\xi_i, i = 0, 1, \ldots, q^{2n}$.
\par Now embed $\PG(4n - 1, q)$ in a $\PG(4n, q)$ and construct a point-line geometry $T(n, 2n, q) = T(O)$ as follows.
\par Points are of three types:
\begin{itemize}
\item [(i)] The points of $\PG(4n, q) \setminus \PG(4n - 1, q)$.
\item [(ii)] The $3n$-dimensional subspaces of $\PG(4n, q)$ which intersect $\PG(4n - 1, q)$ in one of the spaces $\tau_i$.
\item [(iii)] The symbol $(\infty)$.
\end{itemize}
\par Lines are of two types:
\begin{itemize}
\item [(a)] The $n$-dimensional subspaces of $\PG(4n, q)$ which intersect $\PG(4n - 1, q)$ in one of the spaces $\xi_i$.
\item [(b)] The elements of $O$.
\end{itemize}
\par Incidence in $T(n, 2n, q)$ is defined as follows. A point of type (i) is incident only with lines of type (a); here the incidence is that of $\PG(4n, q)$. A point of type (ii) is incident with all lines of type (a) contained in it and with the unique element of $O$ contained in it. The point $(\infty)$ is incident with no line of type (a) and with all lines of type (b).
\bigskip

\begin{theorem}
\label{thm7.1}
The point-line geometry $T(n, 2n, q) = T(O)$ is a GQ of order $(q^n, q^{2n})$.
\end{theorem}
\bigskip

\par An important class of GQ is the class of {\em translation generalized quadrangles}; see Payne and Thas \cite {P&JT:09} and Thas, Thas and Van Maldeghem \cite {TTVM:06}.
\bigskip

\begin{theorem} [8.7.1 of Payne and Thas \cite  {P&JT:09}]
\label{thm7.2}
The point-line geometry $T(O)$ is a TGQ of order $(q^n, q^{2n})$. Conversely, every TGQ of order $(s, s^2), s \not= 1$, is equivalent to a $T(O) = T(n, 2n, q)$, with $s = q^n$.
\end{theorem}
\bigskip

\par It follows that the theory of the TGQ of order $(s, s^2), s \not=1$, is equivalent to the theory of the generalized ovoids in $\PG(4n - 1, q)$, with $s = q^n$.
\par The generalized ovoid $O$ of $\PG(4n - 1, q)$ is a pseudo-quadric if and only if the TGQ $T(O)$ is classical, that is, if and only if $T(O)$ is isomorphic to the GQ arising from an elliptic quadric in $\PG(5, q^n)$.
\bigskip

{\bf History and motivation of our paper}
Let us first define a regular line in a GQ $\mathcal {S} = (P, B, \bf I)$ of order $(s, t)$. Let $M, N \in B$ and $x \in P$ with $M \mathbf {I} x \mathbf {I}  N$. Then we say that $M$ and $N$ are {\em concurrent} lines. Assume $L$ is a given line in $B$ and let $M$ be any line  not concurrent with $L$. There are $s + 1$ lines $U_1, U_2, \ldots, U_{s + 1}$ concurrent with both $L$ and $M$. If for any $M$ there are $s + 1$ lines $V_1 = L, V_2 = M, V_3, \ldots, V_{s +1}$ concurrent with $U_1, U_2, \ldots, U_{s + 1}$ then we say that $L$ is a {\em regular line} of $\mathcal S$.
\par In \cite {JT:23} the following main result was obtained. 
\par {\em Let $q \not= 2$ and assume that $\mathcal {S} = T(O)$ has a regular line $L$ not incident with the point $(\infty)$; say $L$ contains the element $\xi_0$ of the generalized ovoid $O$. Then $O$ is good at $\xi_0$ and relying on results about good generalized ovoids we finally obtain:
\begin{itemize}
\item [(1)] If $q$ is odd, then $\mathcal S$ is the point-line dual of the translation dual of a semifield flock translation generalized quadrangle.
\item [(2)] If $q$ is even, then $T(O)$ is classical.
\end{itemize}}
\par Now some words on the technique used to prove these results. First we projected $O$ from $\xi_0 \in O$ onto a $\PG(3n - 1, q)$ skew to $\xi_0$ and contained in the $\PG(4n - 1, q)$ generated by $O$. This yielded a set $S$ of $q^{2n}$ $(n - 1)$-spaces $\alpha_1, \alpha_2, \ldots,, \alpha_{q^{2n}}$ which, by the regularity of $L$, was proved to satisfy conditions (R) and (T1) of Sections 2 and 3. From this we showed That $O$ is good at $\xi_0$. So for $q$ odd we have (1) but for $q$ even we had to rely again on the regularity of $L$ to prove that $\mathcal {S} = T(O)$ is classical.
\par We remark that the results of \cite {JT:23} are related to the "Moufang Theorem" of Tits \cite {JT:23} and the theorem of Fong and Seitz \cite {JT:23} on groups with a $BN$-pair of rank 2.
\par From this came the idea to consider sets $S$
\begin{itemize}
\item [(i)] which are independent of GQ,
\item [(ii)] which also satisfy (R) and/or (T1),
\item [(iii)] over any field {\bf K},
\item [(iv)] where the dimension $m - 1$ of the set $\xi = \PG(3n - 1, q) \setminus (\alpha_1 \cup \alpha_2 \cup \ldots \cup \alpha_{q^{2n}})$ is independent of the dimension $n - 1$ of $\alpha_i$.
\end{itemize}
\bigskip

\par Now we mention some direct corollaries of Section 6 and of results on the classification of TGQ of order $(s, s^2), s \not= 1$, arising from good ovoids in $\PG(4n - 1, q)$, with $s = q^n$.
\par Let $\mathcal {S} = T(O)$ be a TGQ arising from a generalized ovoid $O = \lbrace \xi_0, \xi_1, \ldots, \xi_{q^{2n}} \rbrace$ in $\PG(4n - 1, q), q \not= 2$. Project from a given $\xi_i$, say $\xi_0$, onto a $\PG(3n - 1, q) \subset \PG(4n - 1, q) $, let $\alpha_j$, with $j \not= 0$, be the projection of $\xi_j$, and let $\xi$ be the projection of $\tau_0 \setminus \xi_0$ with $\tau_0$ the tangent space of $O$ at $\xi_0$. Let $S = \lbrace \alpha_1, \alpha_2, \ldots, \alpha_{q^{2n}} \rbrace$.
\bigskip

\begin{theorem}
\label{thm7.3}
Assume that $q \not= 2$ and that $S$ is a generalized affine space which satisfies either (T1) or (T2), or for which $O$ satisfies (E) for each triple $\lbrace \xi_0, \xi_i, \xi_j \rbrace$. Then $O$ is good at $\xi_0$. Hence for $q$ odd $\mathcal S$ is the point-line dual of the translation dual of a semifield flock TGQ (and so for $q$ odd $O$ does not satisfy (E) for each  triple $\lbrace \xi_0, \xi_i, \xi_j \rbrace$).
\end{theorem}
{\em Proof} 
Follows from Section 6. \eop \\
\bigskip

\begin{theorem}
\label{thm7.4}
Let $q \not= 2$ and assume that for each $\xi_i \in O$ the corresponding set $S$ is a generalized affine space which satisfies either (T1) or (T2), or for which (E) is satisfied for each triple $\lbrace \xi_j, \xi_k, \xi_l \rbrace \subset O$. Then $O$ is good at each element. If $q$ is odd, then $O$ is classical and so $T(O)$ is classical. If $q$ is even, then $O$ is regular, and so $T(O) \cong T(O^\prime)$ with $O^\prime$ an ovoid in a space $\PG(3, q^n)$. Condition (E) is satisfied for each triple of $O$ if and only if $O$ is regular with $q$ even.
\end{theorem}
{\em Proof}
Follows from Section 6. \eop \\
\bigskip

\begin{remark}
\label{rem7.5}
If   $T(O) \cong T(O^\prime)$ with $q$ even and $O^\prime$ an ovoid of $\PG(3, q^n)$ which is not an elliptic quadric, then the only regular lines of $T(O)$ are the lines incident with $(\infty)$; see \cite {P&JT:09}.
\end{remark}

\section {Ideas for further research}
{\bf The case $q = 2$}

In most of the results we have to assume that $q \not= 2$, which is very frustrating. The main reason is that in \ref{thm2.2} we have to assume that $\vert {\bf{K}} \vert > 2$. Notice that for $\vert {\bf{K}} \vert = 2$ condition (R) is trivially satisfied and that conditions (T1) and (E) can be formulated without problems.
\par To extend our results to the case $\bf {K}$ = GF(2) we should prove either that the set $\widetilde S$ of size $2^{n -1}(2^n -1)$ in Section 4 on projections is contained in a $\PG(2n - 2, 2)$, or that the set $S^\prime$ is contained in a $\PG(2n - 1, 2)$.
\bigskip

{\bf Condition (R)}

Let us consider a space $\PG(m + n - 1, {\bf K})$ with an affine spread. Assume that the corresponding set $S$ satisfies the regulus condition, that is, satisfies condition (R); see Section 2. Then $S$ is a generalized affine space.

{\em Problem}: Classify these generalized affine spaces.
\par For $\vert {\bf K} \vert > 2$, strong tools are the existence of the projectivity $\rho_{kl}$ and the fact that induction is possible.

{\bf Condition (E) and condition (P)}

Let $O = \lbrace \xi_0, \xi_1, \ldots, \xi_{q^{2n}} \rbrace$ be a generalized ovoid in $\PG(4n - 1, q)$.
\par Does $q$ even imply that $O$ satisfies condition (E) for each triple $\lbrace \xi_i, \xi_j, \xi_k \rbrace \subset O$, and conversely ?
\par If $O$ satisfies condition (E) for each triple $\lbrace \xi_i, \xi_j, \xi_k \rbrace$, with $\xi_i$ fixed, does $q$ be even and $O$ be good at $\xi_i$ ?
\par Consider $\xi_i, \xi_j, \xi_k$ in $O$, with $i, j, k$ distinct. Let $\langle \xi_i, \xi_j, \xi_k \rangle = \theta, \tau_i \cap \tau_j \cap \theta = \eta_k, \tau_i \cap \tau_k \cap \theta = \eta_j$  and $\tau_j \cap \tau_k  \cap \theta = \eta_i$. Then we say that $O$ satisfies {\em condition (P)} at $\lbrace \xi_i, \xi_j, \xi_k \rbrace$ if $\langle \xi_i, \eta_i \rangle, \langle \xi_j, \eta_j \rangle, \langle \xi_k, \eta_k \rangle$ have a $(n - 1)$-space in common.
\par If $O$ satisfies condition (P) for each triple $\lbrace \xi_i, \xi_j, \xi_k \rangle$, does $q$ be odd and $O$ be classical ?
\par If $O$ satisfies condition (P) for each triple $\lbrace \xi_i, \xi_j, \xi_k \rangle \subset O$ with $\xi_i$ fixed, does $O$ be good at $\xi_i$ which would imply that $q$  is odd ?
\par We remark that condition (P) for generalized ovals in $\PG(3n - 1, q)$, $q$ odd, was already considered in \cite {JT:11}.

\end {document}